\documentclass{article}
\usepackage[a4paper, total={6in, 8in}]{geometry}
\usepackage{amsmath}
\usepackage{amssymb}
\usepackage{amsthm}

\usepackage{graphicx}

\usepackage{xcolor}

\newtheorem{lemma}{Lemma}

\newtheorem{proposition}{Proposition}
\newtheorem{remark}{Remark}
\newcommand{\pl}{\partial}

\usepackage{authblk}

\graphicspath{{./figs/}}
 
\title{A note on the adjoint method for neural ordinary differential equation network}

\author[1]{Pipi Hu\thanks{pisquare@microsoft.com}}
\affil[1]{MSR AI4Science, Beijing, China}
\date{\today}

\begin{document}
\maketitle

\begin{abstract}
Perturbation and operator adjoint method are used to give the right adjoint form rigourously. From the derivation, we can have following results:
1) The loss gradient is not an ODE, it is an integral and we shows the reason; 2) The traditional adjoint form is not equivalent with the back propagation results. 3) The adjoint operator analysis shows that if and only if the discrete adjoint has the same scheme with the discrete neural ODE, the adjoint form would give the same results as BP does. 
\end{abstract}


\section{Introduction}
The sensitivity analysis by adjoint method is an old topic in several fields, such as in geophysics, seismic imaging, photonics. Nowadays, the adjoint method is used in the neural ordinary differential equations to calculate the gradients of the cost function with respect to the training parameters. \cite{chen2018neural} gives  ``A Modern Proof of the Adjoint Method" in its appendix. The results are right, but the proof is a little vague. 

The adjoint method actually gives the gradients of the cost which is a Lagrange form of cost not the gradients of the cost directly. The dependence of the variables in the proof is ambiguous and may not be properly identified, which is confusing for understanding the adjoint method in the back propagation. Such these problems prevent the deep understanding and modulating to the existing formulas. Meanwhile, there almost does not exist a complete or direct proof for the adjoint methods in neural ode because of the special properties of the control problems in the neural ode. These reasons drive us to give a rigorous and easy-understanding proof for the adjoint method in such typical problem, by which, deeper works for modulating the details in the adjoint form can be done. Through the rigorous analysis, we can clarify that the gradients to the loss of multi labels sampled on one trajectory have different form from that in \cite{chen2018neural}.  

 By the calculus of variation used in optimal control we can get the right form of odes of adjoint variables and hence we can derive the gradients of the cost with respect to the control parameters (called training parameters in neural ode). We break this note into four main parts to introduce the results. The background of the neural ode is omitted and if not familiar, see the first two sections of \cite{chen2018neural}.
 
 Section 2 gives the setup of the optimization problem for the neural ode. In this section, we will exhibit the typical properties of the control problem for neural ode. In Section 3, we will give a rigorous proof of adjoint method by calculus of variation. The calculus of variation is generally used to derive the results of optimal control in \cite{liberzon2011calculus}. By a modulation to the technologies in \cite{liberzon2011calculus}, we can have a deeper comprehension and clear proof of the adjoint method. Then, we will show that for the neural ode, the property of time-independent control parameters makes the form of adjoint variable special. Actually, the first part of proof for adjoint method in \cite{chen2018neural} is using such special form of the adjoint variable to prove the adjoint dynamics. We would show that this definition is rigorous, but the derivation is ambiguous. Section 4 introduces a simple but not rigorous proof of the adjoint method by direct partial differential calculations. We want to present that, with much simpler differentiations, we can still get the right form of adjoint method but the proof itself is ambiguous with confusing variable dependence. In section 5, we claim that in the neural ode when the loss is calculated by multi labels in the same trajectory, the new gradients of the loss is slightly different from that in \cite{chen2018neural}.
 
 The loss integral is not an ODE which can be easily reduced from the derivation and it also stated in \cite{kidger2020hey} but not give an explanation.

 Besides the main part, we will show in appendix a version of the discrete form of the adjoint method  which is a rigorous restatement of the results in \cite{zhuang2019ordinary}. In the work of \cite{hu2020revealing}, we didn't use the adjoint method for the gradients of cost with respect to the training coefficients because of the instabilities of the adjoint methods. This discrete version takes the full advantages of the forward propagation, making back propagation stable as long as the stability holds for forward propagation. More works may be done by such contributions depending on the right comprehensions of the adjoint method.
 
\section{Problem setup}
We abstract the optimization problem of the neural ode as follows:
\begin{equation}
\label{eq.lo}
\begin{aligned}
\min_{\theta} \quad & \mathcal{L}(t_0,t_1,\theta(t),z(t_0)) = L(z(t_f))\\
\textrm{s.t.} \quad & \dot{z}(t)=f(t,z(t),\theta(t)),\\
\end{aligned}
\end{equation}
where $\mathcal{L}\in\mathbb{R}$ is a cost function (also called loss in machine learning), $\theta(t)\in \mathbb{R}^P$ is the control parameter (also called training parameters in machine learning) and $z(t_0)\in\mathbb{R}^N$ is the initial value at time $t_0$. $P$ is the number of the parameters and $N$ is the dimension of the trajectory of $z(t)$.  In typical optimal control theory, the cost $ \mathcal{L}$ can have running cost and terminal cost, but in this problem there only exists the terminal cost $L(z(t_f))$ which is always a distance between the labels and the predictions in the neural network. 

Further, the optimization problem \eqref{eq.lo} has the following properties.
\begin{enumerate}
\item {\bf The initial value} $z(t_0)$ of the trajectories is independent on the control parameters $\theta(t)$, i.e., for $\forall \theta\in \Theta$, $z(t_0)$ is given as a constant, where $\Theta$ is the possible parameter space. 
\item The initial time $t_0$ is given, and any trajectory perturbed by $\theta$ all ends at time $t_f$, i.e., the problem is {\bf fixed end time $t_f$} with {\bf free end point} $z(t_f)$.
\item $\theta(t)\equiv \theta$, i.e., $\theta(t)$ {\bf doesn't change} with time $t$. 
\end{enumerate}

For such problem, the adjoint method gives the following results.

\begin{proposition} Adjoint method:
The adjoint variable $a(t)$ satisfies an back propagation ODE
\begin{equation}
\dot{a}(t) = -a(t)\frac{\partial f}{\partial z}|_*, \, \text{with } a(t_f) = -\frac{\partial L}{\partial z(t_f)}|_*,
\end{equation}
and the gradients of the Lagrange cost $J(\theta) = L(z(t_f)) + \int_{t_0}^{t_f} a(t) \bigl(\dot{z}(t)-f(t,z(t),\theta(t))\bigr)dt$ has the form
\begin{equation}
\frac{dJ(\theta)}{d\theta} = -\int_{t_0}^{t_f}a(t)\frac{\partial f}{\partial \theta}|_*dt.
\end{equation}
\end{proposition}

In the next section, we will give a rigorous proof of the adjoint method for \eqref{eq.lo} with the three properties. For a deeper understanding, we first solve the optimization problem of \eqref{eq.lo} with the first two properties, i.e., we assume $\theta(t)$ depends on time $t$. And then we use the third property to show that how this property makes the problem different from traditional optimal control problem for a better comprehension of adjoint method in neural ode.

Before the proof, we make the agreement of the dimension of the matrix generated from partial differentiations.
We treat vector as matrix with its first or the second dimension one. For example, a $N$-size row vector is a $1\times N$ matrix and $N$-size column vector is a $N\times 1$ matrix. For the partial differentiations, $A_{ij}=\frac{\partial g_i}{\partial x_j}$ is in the $i^{\text{th}}$ row and $j^{\text{th}}$ column of the $|g|\times |x|$-size matrix A, where $i=1,2,\cdots, |g|, \, j=1,2,\cdots, |x|$,  and $|g|,$  $|x|$ represent the sizes of vector $g$ and $x$ respectively. Particularly, $A_{i}=\frac{\partial g_i}{\partial x}$ represents a column vector when $|x|=1$ and $A_{i}=\frac{\partial g}{\partial x_i}$ represent a row vector when $|g|=1$. They also represent $|g|\times 1$-size matrix and $1\times |x|$-size matrix in the view of matrix.

\section{A rigorous proof of the adjoint method}
The general idea is that we do the small perturbations on the control parameters $\theta$ to see how the perturbations influence the cost function. 

Let 
\begin{equation}
\theta = \theta^*+\epsilon\zeta,
\end{equation}
where $\theta^* = \theta^*(t)$ is the original control parameter to be perturbed and $\theta = \theta(t)$ is the perturbed parameter, and $\zeta=\zeta(t)$ is a perturbation with a small perturbation factor $\epsilon$.

The first question arises, how is the perturbed trajectory $z(t)$ to the original one $z^*(t)$?
Here we have 
\begin{equation}
z(t) = z(t)^* + \epsilon\eta(t) + \mathcal{O}(\epsilon^2).
\end{equation}
We can't say that the corresponding perturbations of $z(t)$ is linear, but we can say the nonlinear term is a higher order of $\epsilon^2$ with the perturbation factor $\epsilon$ small enough. The perturbation scheme is shown in Figure \ref{fig.rela}.

\begin{figure}[htp]
\includegraphics[width=\textwidth]{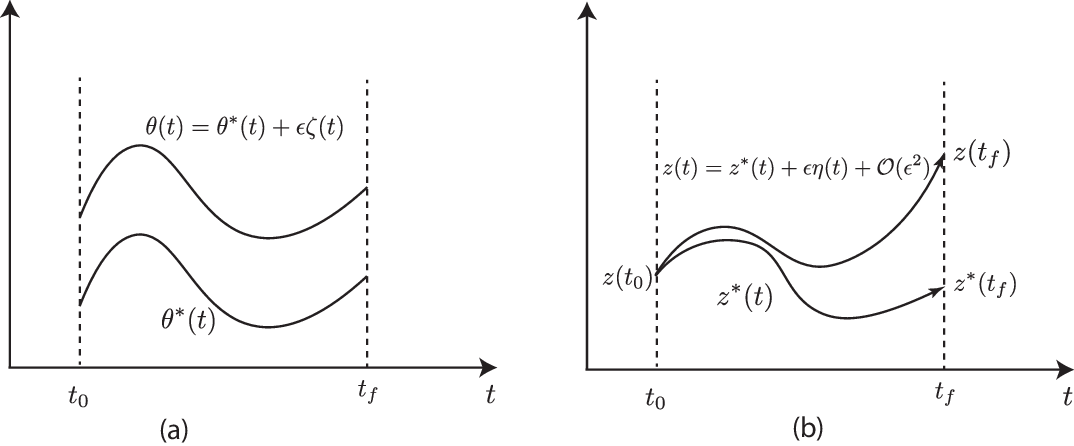}
\caption{The perturbation schemes. (a): the perturbation of $\theta$. (b) the perturbation of $z$ based on the perturbation of $\theta$.}
\label{fig.rela}
\end{figure}

\begin{lemma}
\label{lemma.eta}
The perturbations $\eta(t)$ to the trajectory caused by the perturbations of parameters also satisfies an ODE
\begin{equation}
\dot{\eta}(t) = f_z|_*\eta(t) + f_\theta|_*\zeta(t),
\end{equation}
where $f_z|_* = \frac{\partial f}{\partial z}|_{z = z^*(t)}$ and $f_\theta|_* = \frac{\partial f}{\partial \theta}|_{\theta = \theta^*(t)}$.
\end{lemma}

\begin{proof}
It is easy to check that $\eta(t) = z_\epsilon(t)|_{\epsilon=0}$. Further, we have 
\begin{equation}
\dot{\eta}(t) = z_{t,\epsilon}(t)|_{\epsilon=0} =  z_{\epsilon,t}(t)|_{\epsilon=0} = \frac{d}{d\epsilon}|_{\epsilon=0} z_t(t),
\end{equation}
and recall that the constraint $\dot{z}(t)=f(t,z,\theta)$. It follows that 
\begin{equation}
\dot{\eta}(t) = \frac{d}{d\epsilon}|_{\epsilon=0} f(t,z^*+\epsilon\eta+\mathcal{O}(\epsilon^2),\theta^*+\epsilon\zeta)=
f_z|_*\eta(t) + f_\theta|_*\zeta(t).
\end{equation}
This completes the proof.
\end{proof}

To prove the adjoint method by calculus of variation, we first transform the optimization problem to the Lagrange
form by introducing a general form of Lagrange multiplier $a(t)$ as 
\begin{equation}
\label{eq.l}
\begin{aligned}
\min_{\theta} \quad & J(\theta) = L(z(t_f)) + \int_{t_0}^{t_f} a(t) \bigl(\dot{z}(t)-f(t,z(t),\theta(t))\bigr)dt,\\
\end{aligned}
\end{equation}
and the three properties above still hold.

To clarify the influence on the Lagrange cost function by perturbations of parameters, we calculate the difference directly as 
\begin{equation}
\label{eq.j}
J(\theta)-J(\theta^*) = L(z(t_f)) - L(z^*(t_f)) + \int_{t_0}^{t_f} a(t)(\dot{z}(t)-\dot{z}^*(t))dt - \int_{t_0}^{t_f} a(t)(f(t,z,\theta)-f(t,z^*,\theta^*))dt,
\end{equation}
where  $\theta = \theta^*+\epsilon\zeta$ and $z(t) = z(t)^* + \epsilon\eta(t) + \mathcal{O}(\epsilon^2)$. 
Calculate \eqref{eq.j} term by term. The first term yields
\begin{equation}\label{eq.first}
L(z(t_f)) - L(z^*(t_f)) = L_{z(t_f)}|_*\epsilon\eta(t_f)+\mathcal{O}(\epsilon^2).
\end{equation}
The second term yields
\begin{equation}\label{eq.second}
\begin{aligned}
\int_{t_0}^{t_f} a(t)(\dot{z}(t)-\dot{z}^*(t))dt &= a(t)(z(t)-z^*(t))|_{t_0}^{t_f} - \int_{t_0}^{t_f} \dot{a}(t)(z(t)-z^*(t))dt\\
&= a(t_f)\epsilon\eta(t_f)- \int_{t_0}^{t_f} \dot{a}(t)\epsilon\eta(t)dt+\mathcal{O}(\epsilon^2),\\
\end{aligned}
\end{equation}
where $z(t_0)-z^*(t_0)=0$ based on the first property. 
Similarly, direct calculation of the third term gives 
\begin{equation}\label{eq.third}
- \int_{t_0}^{t_f} a(t)\bigl(f(t,z^*+\epsilon\eta+\mathcal{O}(\epsilon^2),\theta^*+\epsilon\zeta)-f(t,z^*,\theta^*)\bigr)dt
=- \int_{t_0}^{t_f} a(t)\bigl(f_z|_*\epsilon\eta + f_\theta|_*\epsilon\zeta\bigr)dt+\mathcal{O}(\epsilon^2).
\end{equation}

Substitute \eqref{eq.first}-\eqref{eq.third} to  \eqref{eq.j}, the difference of the Lagrange cost  becomes
\begin{equation}
J(\theta) - J(\theta^*) = \epsilon\eta(t_f)\bigl(a(t_f)+L_z|_*\bigr) - \int_{t_0}^{t_f}\bigl(\dot{a}(t)+a(t)f_z|_*\bigr)\epsilon\eta(t)dt - \int_{t_0}^{t_f}a(t)f_\theta|_*\epsilon\zeta(t)dt+\mathcal{O}(\epsilon^2).
\end{equation}

Till now, we have nothing constraints on the Lagrange multiplier $a(t)$. To derive a compact form of the gradients of the Lagrange cost, let 
\begin{equation}
\dot{a}(t) = -a(t)f_z|_*=-a(t)\frac{\partial f}{\partial z}|_*, \, \text{ with } a(t_f) = -L_{z(t_f)}|_*=-\frac{\partial L}{\partial z(t_f)}|_*,
\end{equation}
which determines a trajectory of the Lagrange multiplier.

Then the difference of the Lagrange cost becomes 
\begin{equation}
J(\theta) - J(\theta^*) = -\int_{t_0}^{t_f} a(t)f_\theta|_*\epsilon\zeta(t)dt.
\end{equation}
{\bf \color{red} It seems there is no need for the constant of $\theta$?}
Further, from the third property $\zeta(t)\equiv\zeta$ is a constant along time, i.e., $\theta(t)\equiv\theta$. Then we have 
\begin{equation}
\frac{dJ(\theta)}{d\theta} = \lim_{\epsilon\to 0} \frac{J(\theta)-J(\theta^*)}{\epsilon\zeta} = -\int_{t_0}^{t_f}a(t)f_\theta|_*dt=-\int_{t_0}^{t_f}a(t)\frac{\partial f}{\partial \theta}|_*dt.
\end{equation}
This completes the proof.

From the proof of the adjoint method, we can find that the constant property makes the partial differentiations easily obtained, which makes the problem distinct with the general control problems.

As the Lagrange multiplier $a(t)$ satisfies the following ODE
\begin{equation*}
\dot{a}(t) = -a(t)f_z|_* = -f_z^T|_*a(t).
\end{equation*}
Recall that $\eta(t)$ satisfies 
\begin{equation*}
\dot{\eta}(t) = f_z|_*\eta(t) + f_\theta|_*\zeta(t).
\end{equation*}

Two linear systems $\dot{x}=Ax$ and $\dot{z}=-A^Tz$ are said to be adjoint to each other, and thus for this reason the multiplier $a(t)$ is said to be an adjoint variable.

\begin{lemma}\label{lemma.a}
For the control problem of the neural ode, the multiplier $a(t)$ chosen above have the following form
\begin{equation}
a(t) = - \frac{\partial L}{\partial z}.
\end{equation}
\end{lemma}
\begin{proof}
Let 
\begin{equation}\label{eq.dl}
\lambda(t) = - \frac{\partial L}{\partial z},
\end{equation}
and we can see that the $\lambda$ is actually the function of both $t$ and $\theta$.
 
And the derivation to t shows that
\begin{equation}\label{eq.lambda}
\frac{d\lambda}{dt} = \frac{\partial \lambda}{\partial t} + \frac{\partial \lambda}{\partial \theta}\frac{\partial \theta}{\partial t} =  \frac{\partial \lambda}{\partial t}=\lim_{\epsilon\to 0} \frac{\lambda(t+\epsilon)-\lambda(t)}
{\epsilon}.
\end{equation}

And 
\begin{equation}
\lambda(t+\epsilon) = -\frac{\partial L}{\partial z(t+\epsilon)},\quad \lambda(t) = -\frac{\partial L}{\partial z(t)} = -\frac{\partial L}{\partial z(t+\epsilon)}\frac{\partial z(t+\epsilon)}{\partial z(t)},
\end{equation}
thus we have 
\begin{equation}\label{eq.lambdad}
\lambda(t+\epsilon)-\lambda(t)=\frac{\partial L}{\partial z(t+\epsilon)}\biggl(\frac{\partial z(t+\epsilon)}{\partial z(t)}-I\biggr).
\end{equation}

Recall 
\begin{equation}
\dot{z}(t) = f(t,\theta,z),
\end{equation}
we have 
\begin{equation}
z(t+\epsilon) = \int_t^{t+\epsilon} f(t,\theta,z)dt + z(t).
\end{equation}
It follows that 
\begin{equation}\label{eq.zep}
\frac{\partial z(t+\epsilon)}{\partial z(t)} = f_z(t,\theta,z)\epsilon+I +\mathcal{O}(\epsilon^2).
\end{equation}

Substitute \eqref{eq.zep} into \eqref{eq.lambdad} and combine the result with \eqref{eq.lambda} we can finally reach
\begin{equation}
\dot{\lambda}(t) = \lim_{\epsilon\to0} \frac{\partial L}{\partial z(t+\epsilon)}\frac{\partial f}{\partial z}+\mathcal{O}(\epsilon^2)=-\lambda(t)\frac{\partial f}{\partial z}.
\end{equation}
Further the definition of $\lambda(t)$ in \eqref{eq.dl} implies 
\begin{equation}
\lambda(t_f) = -\frac{\partial L}{\partial z(t_f)}.
\end{equation}
Till now we have proved that $\lambda(t)$ and $a(t)$ satisfies the same initial value problem of ODE and by the existence and uniqueness of the ODE, we can say that $\lambda(t)=a(t)$. This completes the proof.
\end{proof}
Lemma \ref{lemma.a} tells us that the adjoint variable can be a form of partial differentiations of cost to the trajectory variable.

\section{A direct and simple but not rigorous proof of the adjoint method}
In this section, the gradients of Lagrange cost function $J(\theta)$ with respect to the parameters $\theta$ are directly calculated by partial differentiations.

Recall that 
\begin{equation}\label{eq.J}
\begin{aligned}
J(\theta) &= L(z(t_f)) + \int_{t_0}^{t_f} a(t) \bigl(\dot{z}(t)-f(t,z(t),\theta(t))\bigr)dt,\\
&= L(z(t_f)) + a(t)z(t)|_{t_0}^{t_f} - \int_{t_0}^{t_f} \dot{a}(t)z(t)dt - \int_{t_0}^{t_f} a(t) f(t,z(t),\theta(t))dt.\\
\end{aligned}
\end{equation}
Perform the differentiations to $\theta$ on both sides of \eqref{eq.J} and we have 
\begin{equation}
\begin{aligned}
\frac{dJ(\theta)}{\partial \theta} &= \frac{\partial L}{\partial z(t_f)} \frac{\partial z(t_f)}{\partial \theta} + \frac{d}{d\theta} a(t_f)z(t_f) - \frac{d}{d\theta} a(t_0)z(t_0) - \int_{t_0}^{t_f}\dot{a}(t)\frac{\partial z}{\partial \theta} dt - \int_{t_0}^{t_f} a(t)\bigl(\frac{\partial f}{\partial z}\frac{\partial z}{\partial \theta}+\frac{\partial f}{\partial \theta}\bigr)dt\\
& = \frac{\partial z(t_f)}{\partial\theta}\bigl(a(t_f)+\frac{\partial L}{\partial z(t_f)}\bigr) - \int_{t_0}^{t_f}\bigl(\dot{a}(t)+a(t)\frac{\partial f}{\partial z}\frac{\partial z}{\partial\theta}\bigr)dt - \int_{t_0}^{t_f} a(t)\frac{\partial f}{\partial \theta}dt,
\end{aligned}
\end{equation}
where we have used the fact $\frac{d}{d\theta} a(t_0)z(t_0)=0$ based on the first property.

Let $\dot{a}(t)=-a(t)\frac{\partial f}{\partial z(t)}$ with $a(t_f)=-\frac{\partial L}{\partial z(t_f)}$, we have the compact form of the gradients as 
\begin{equation}
\frac{d J(\theta)}{d\theta} = -\int_{t_0}^{t_f} a(t)\frac{\partial f}{\partial \theta}dt.
\end{equation}
This completes the proof.

This proof is ambiguous about the dependence of the variables. If no more insight into this problem, it may cause errors. But it is indeed a fast approach to the right form of the adjoint method.

\section{The optimization of the loss with multi labels}
There are always multi labels on the trajectory, e.g., we are given the values of several time snapshots along the trajectory which can be used as labels to compute the total loss.

The problem reads 
\begin{equation}
\label{eq.ml}
\begin{aligned}
\min_{\theta} \quad & \mathcal{L}(t_0,t_1,\cdots,t_N,\theta(t),z(t_0)) = L(z(t_1),z(t_2),\cdots,z(t_N))\\
\textrm{s.t.} \quad & \dot{z}(t)=f(t,z(t),\theta(t)),\\
\end{aligned}
\end{equation}
where the only difference from \eqref{eq.lo} is the dependence of loss $L$ on $z(t_1),z(t_2),\cdots,z(t_N)$ not just on $z(t_N)$ (denoted as $z(t_f)$ in \eqref{eq.lo}). The three properties also hold with terminal cost changing into a general function of several points on the trajectory.

By introducing the Lagrange multiplier $a(t)$ as before, the optimization problem \eqref{eq.ml} becomes
\begin{equation}
\min_{\theta} J(\theta) = L(z(t_1),z(t_2),\cdots,z(t_N)) +  \int_{t_0}^{t_N} a(t)(\dot{z}(t)-f(t,z(t),\theta(t))).
\end{equation}
The schemes of this optimization problem are shown in Figure \ref{fig.rela2}.

\begin{figure}[htp]
\includegraphics[width=\textwidth]{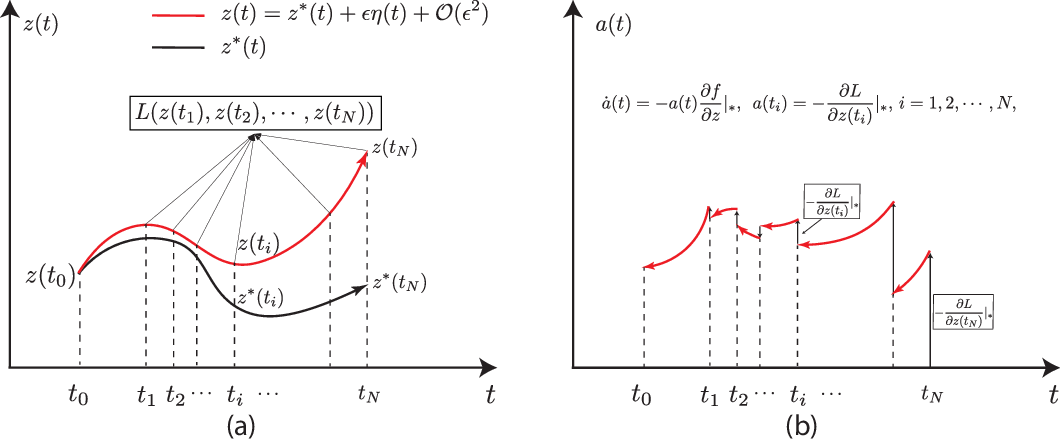}
\caption{The perturbation schemes. (a): the perturbation of $\theta$. (b) the perturbation of $z$ based on the perturbation of $\theta$ and the times of multi labels are denoted. $t_i$ represent the $i^{th}$ label on the time line with the corresponding trajectory value $z(t_i)$.}
\label{fig.rela2}
\end{figure}
 
 We do the perturbation as before, and the difference of the Lagrange cost have the form
 \begin{equation}
  \begin{aligned}
 J(\theta) - J(\theta^*) = &\bigl(L(z(t_1),z(t_2)\cdots,z(t_N)) - L(z^*(t_1),z^*(t_2)\cdots,z^*(t_N))\bigr)\\ 
 &+ \sum_{i=1}^N\int_{t_0}^{t_i} a(t)\bigl(\dot{z}(t)-\dot{z}^*(t)\bigr)dt \\
 &-\sum_{i=1}^N\int_{t_0}^{t_i} a(t)\bigl(f(t,z,\theta)-f(t,z^*,\theta^*)\bigr)dt.
  \end{aligned}
 \end{equation}
 
 The first term yields
\begin{align}
\begin{split}
\label{eq.first}
L(z(t_1)\cdots,z(t_N)) - L(z^*(t_1)\cdots,z^*(t_N)) &= \sum_{i=1}^NL_{z(t_i)}|_*\epsilon\eta(t_i)+\mathcal{O}(\epsilon^2)\\
&=\int_{t_0}^{t_N}\sum_{i=1}^NL_{z(t_i)}|_*\epsilon\eta(t)\delta(t-t_i)dt+\mathcal{O}(\epsilon^2).
\end{split}
\end{align}
The second term yields
\begin{equation}\label{eq.second}
\begin{aligned}
\int_{t_0}^{t_N} a(t)(\dot{z}(t)-\dot{z}^*(t))dt &= a(t)(z(t)-z^*(t))|_{t_0}^{t_N} - \int_{t_0}^{t_N} \dot{a}(t)(z(t)-z^*(t))dt\\
&= - \int_{t_0}^{t_N} \dot{a}(t)\epsilon\eta(t)dt+\mathcal{O}(\epsilon^2),\\
\end{aligned}
\end{equation}
where we have used the fact $z(t_0)-z^*(t_0)=0$ and the presume 
\begin{equation}
a(t_N)=0.
\end{equation}
Similarly, direct calculation of the third term gives 
\begin{equation}\label{eq.third}
- \int_{t_0}^{t_N} a(t)\bigl(f(t,z^*+\epsilon\eta+\mathcal{O}(\epsilon^2),\theta^*+\epsilon\zeta)-f(t,z^*,\theta^*)\bigr)dt
=- \int_{t_0}^{t_N} a(t)\bigl(f_z|_*\epsilon\eta + f_\theta|_*\epsilon\zeta\bigr)dt+\mathcal{O}(\epsilon^2).
\end{equation}

 Compute the difference term by term as in  Chapter 3, we obtain 
 \begin{equation}
 \begin{aligned}
 J(\theta) - J(\theta^*) =  &- \int_{t_0}^{t_N}\bigl(\dot{a}(t)+a(t)\frac{\partial f}{\partial z}|_*-\sum_{i=1}^N\frac{\partial L}{\partial z(t_i)}|_*\delta(t-t_i)\bigr)\epsilon \eta(t)dt \\
 & -\int_{t_0}^{t_N}a(t)\frac{\partial f}{\partial \theta}|_*\epsilon\zeta(t)dt+ \mathcal{O}(\epsilon^2).\\ 
 \end{aligned}
 \end{equation}
 
Let the adjoint equation as
\begin{equation}\label{eq.adjoint}
\dot{a}(t)=-a(t)\frac{\partial f}{\partial z}|_*, \quad \text{with } a^{initial}(t_i)=a(t_i)-\frac{\partial L}{\partial z(t_i)}|_*,\, i=N,N-1,\cdots 1,
\end{equation}
we have the compact form of the difference as 
\begin{equation}
J(\theta) - J(\theta^*) =  -\int_{t_0}^{t_N}a(t)\frac{\partial f}{\partial \theta}|_*\epsilon\zeta(t)dt + \mathcal{O}(\epsilon^2).
\end{equation}
Further when $\theta(t)\equiv\theta$, i.e., $\zeta(t)\equiv \zeta$, we have the gradients of the Lagrange loss as 
\begin{equation}\label{eq.lossg}
\frac{d J(\theta)}{d \theta} =  -\int_{t_0}^{t_N}a(t)\frac{\partial f}{\partial \theta}|_*dt.
\end{equation}
Such form is the same with that in \cite{chen2018neural}.

\section{The optimization of the loss with multi labels (Another Form)}
There are always multi labels on the trajectory, e.g., we are given the values of several time snapshots along the trajectory which can be used as labels to compute the total loss.

The problem reads 
\begin{equation}
\label{eq.ml}
\begin{aligned}
\min_{\theta} \quad & \mathcal{L}(t_0,t_1,\cdots,t_N,\theta(t),z(t_0)) = L(z(t_1),z(t_2),\cdots,z(t_N))\\
\textrm{s.t.} \quad & \dot{z}(t)=f(t,z(t),\theta(t)),\\
\end{aligned}
\end{equation}
where the only difference from \eqref{eq.lo} is the dependence of loss $L$ on $z(t_1),z(t_2),\cdots,z(t_N)$ not just on $z(t_N)$ (denoted as $z(t_f)$ in \eqref{eq.lo}). The three properties also hold with terminal cost changing into a general function of several points on the trajectory.

By introducing the Lagrange multiplier $a(t)$ as before, the optimization problem \eqref{eq.ml} becomes
\begin{equation}
\min_{\theta} J(\theta) = L(z(t_1),z(t_2),\cdots,z(t_N)) + \sum_{i=1}^N \int_{t_0}^{t_i} a(t)(\dot{z}(t)-f(t,z(t),\theta(t))).
\end{equation}

 We do the perturbation as before, and the difference of the Lagrange cost have the form
 \begin{equation}
  \begin{aligned}
 J(\theta) - J(\theta^*) = &L\bigl(z(t_1),z(t_2)\cdots,z(t_N)\bigr) - L\bigl(z^*(t_1),z^*(t_2)\cdots,z^*(t_N)\bigr)\\ 
 &+ \sum_{i=1}^N\int_{t_0}^{t_i} a(t)\bigl(\dot{z}(t)-\dot{z}^*(t)\bigr)dt \\
 &-\sum_{i=1}^N\int_{t_0}^{t_i} a(t)\bigl(f(t,z,\theta)-f(t,z^*,\theta^*)\bigr)dt.
  \end{aligned}
 \end{equation}
 
 Compute the difference term by term as in  Chapter 3, we obtain 
 \begin{equation}
 \begin{aligned}
 J(\theta) - J(\theta^*) = & \sum_{i=1}^N\bigl(a(t_i)+\frac{\partial L}{\partial z(t_i)}|_*\bigr)\epsilon \eta(t) \\
 &- \sum_{i=1}^N\int_{t_0}^{t_i}\bigl(\dot{a}(t)+a(t)\frac{\partial f}{\partial z}|_*\bigr)\epsilon \eta(t)dt \\
 & -\sum_{i=1}^N\int_{t_0}^{t_i}a(t)\frac{\partial f}{\partial \theta}|_*\epsilon\zeta(t)dt+ \mathcal{O}(\epsilon^2).\\ 
 \end{aligned}
 \end{equation}
 
Let the adjoint equation as
\begin{equation}\label{eq.adjoint}
\dot{a}(t)=-a(t)\frac{\partial f}{\partial z}|_*, \quad \text{with } a(t_i)=-\frac{\partial L}{\partial z(t_i)}|_*,\, i=1,2,\cdots, N,
\end{equation}
we have the compact form of the difference as 
\begin{equation}
J(\theta) - J(\theta^*) =  -\sum_{i=1}^N\int_{t_0}^{t_i}a(t)\frac{\partial f}{\partial \theta}|_*\epsilon\zeta(t)dt + \mathcal{O}(\epsilon^2).
\end{equation}
Further when $\theta(t)\equiv\theta$, i.e., $\zeta(t)\equiv \zeta$, we have the gradients of the Lagrange loss as 
\begin{equation}\label{eq.lossg}
\frac{d J(\theta)}{d \theta} =  -\sum_{i=1}^N\int_{t_0}^{t_i}a(t)\frac{\partial f}{\partial \theta}|_*dt.
\end{equation}

The adjoint equation \eqref{eq.adjoint} shows that the adjoint variable $a(t)$ can be computed by a differential equation from $t_N$ to $t_0$ but the $a(t)$ is piecewise computed in each interval of $[t_{i-1},t_{i}]$ for $i=N,N-1,\cdots,1$ with different initial value $-\frac{\partial L}{\partial z(t_i)}|_*$. And the gradients of the loss \eqref{eq.lossg} shows that the gradients is not just the simple integration about the adjoint variable, but a summation of different parts with a coefficient $N-i+1$, i.e., the loss gradient \eqref{eq.lossg} read 
\begin{equation}
\frac{d J(\theta)}{d \theta} =  -\sum_{i=1}^N\int_{t_0}^{t_i}a(t)\frac{\partial f}{\partial \theta}|_*dt = -\sum_{i=N}^1 (N-i+1)\int_{t_{i-1}}^{t_i} a(t)\frac{\partial f}{\partial \theta}|_*dt.
\end{equation}
Such form is different from that in \cite{chen2018neural}.

\section{The adjoint operator derivation for the neural ordinary differential equations}

In this section, the operator adjoint method is performed on the neural ordinary differential equations and we will show that the term ``adjoint" may be more related with the operator adjoint in the adjoint methods.

\subsection{General statements of the adjoint operator derivation}

The adjoint method is more related to the adjoint operator and in this subsection, we will show that how the adjoint operator method is performed to derive the loss gradient with respect to the learning parameters. 

Given the general optimization problem

\begin{equation}
\label{eq.general}
\begin{aligned}
\min_{\theta} \quad &L(z)\\
\textrm{s.t.} \quad & \mathcal{T}z=f,\\
\end{aligned}
\end{equation}
where the loss $L(z)$ is a functional of the function $z$, $\mathcal{T} = \mathcal{T}(\theta)$ is the linear operator performed on function $z=z(\theta)$, and $f=f(z,\theta)$ is a nonlinear function.

The loss gradient have the form
\begin{equation}
\frac{\partial L}{\partial \theta} = \langle\frac{\partial L}{\partial z},\frac{\partial z}{\partial \theta}\rangle,
\end{equation}
where $\frac{\partial L}{\partial z}$ can be derived simply and the inner product is done in the space where $z$ exists. 

Perform the derivation to $\theta$ on the constraint equation of the optimization problem \eqref{eq.general}, we have 
\begin{equation}
\frac{\pl\mathcal{T}}{\pl\theta} z + \mathcal{T} \frac{\pl z}{\pl\theta}=\frac{\pl f}{\pl z}\frac{\pl z}{\pl \theta} + \frac{\pl f}{\pl \theta},
\end{equation}
and it follows that 
\begin{equation}
(\mathcal{T} -\frac{\pl f}{\pl z} ) \frac{\pl z}{\pl\theta}=-\frac{\pl\mathcal{T}}{\pl\theta} z+\frac{\pl f}{\pl \theta},
\end{equation}

Suppose there exist an adjoint variable $\lambda$ which satisfies
\begin{equation}\label{eq.lambda}
\langle\frac{\partial L}{\partial z},\frac{\partial z}{\partial \theta}\rangle = \langle\lambda,(\mathcal{T}-\frac{\pl f}{\pl z}) \frac{\partial z}{\partial \theta}\rangle.
\end{equation}

Then we have the loss gradient has the following simple form 
\begin{equation}
\frac{\partial L}{\partial \theta}=\langle\frac{\partial L}{\partial z},\frac{\partial z}{\partial \theta}\rangle = \langle\lambda,(\mathcal{T}-\frac{\pl f}{\pl z}) \frac{\partial z}{\partial \theta}\rangle = \langle\lambda,-\frac{\partial \mathcal{T}}{\partial \theta}z+\frac{\pl f}{\pl \theta}\rangle.
\end{equation}

The key problem becomes how to find the $\lambda$.
Recall that $\lambda$ satisfies the \eqref{eq.lambda}, and we have 
\begin{equation}\label{eq.adjoint}
\langle\frac{\partial L}{\partial z},\frac{\partial z}{\partial \theta}\rangle = \langle\lambda,(\mathcal{T}-\frac{\pl f}{\pl z}) \frac{\partial z}{\partial \theta}\rangle = \langle(\mathcal{T}^*-\frac{\pl f}{\pl z})\lambda, \frac{\partial z}{\partial \theta}\rangle,
\end{equation}
where we suppose that the adjoint operator $\mathcal{T}^*$ exists.
Hence, $\lambda$ satisfies the following equation 
\begin{equation}\label{eq.al}
(\mathcal{T}^*-\frac{\pl f}{\pl z})\lambda = \frac{\partial L}{\partial z}.
\end{equation}

The operator adjoint methods for the loss gradient of the optimization problem \eqref{eq.general} has the following form:
\begin{align}
\frac{\partial L}{\partial \theta} = \langle\lambda,-\frac{\partial \mathcal{T}}{\partial \theta}z+\frac{\pl f}{\pl \theta}\rangle,\\
\text{where} \quad (\mathcal{T}^*-\frac{\pl f}{\pl z})\lambda = \frac{\partial L}{\partial z}.
\end{align}

\subsection{Operator adjoint methods for neural ODE}
The adjoint forms can be derived by the operator adjoint methods for neural ODE.  

As the loss function is just the function of $z(t_i), i=1, 2,\cdots N$ where $t_i$ is discrete time point, thus
it is needed to derive the adjoint forms for the neural ODE.

Recall that the optimization of the loss with multi label has the following form
The problem reads 
\begin{equation}
\label{eq.oml}
\begin{aligned}
\min_{\theta} \quad &  L(z(t_1),z(t_2),\cdots,z(t_N))\\
\textrm{s.t.} \quad & \dot{z}(t)=f(t,z(t),\theta(t)).\\
\end{aligned}
\end{equation}

To use the adjoint operator, we use the trick of a delta function $\delta(t)$ with 
$$\int\delta(t)d t = 1. $$

And the loss gradient can be extended into the integral of $t$ as follows 
\begin{align}
\frac{\pl L}{\pl \theta} &= \sum_i^N \frac{\pl L}{\pl z(t_i)}\frac{\pl z(t_i)}{\pl \theta}\\
&= \sum_i^N \frac{\pl L}{\pl z(t_i)}\int\frac{\pl z(t)}{\pl \theta}\delta(t-t_i) dt \\
&=\int \frac{\pl z(t)}{\pl \theta}\sum_i^N \frac{\pl L}{\pl z(t_i)}\delta(t-t_i) dt.
\end{align}

Perform the derivation of $\theta$ on the constraint we have 
\begin{equation}
\frac{d z_\theta}{d t}-\frac{\pl f}{\pl z}z_\theta-\frac{\pl f}{\pl \theta}=0,
\end{equation}
and we have 
\begin{equation}
(\frac{d}{d t}-\frac{\pl f}{\pl z})z_\theta = \frac{\pl f}{\pl \theta}.
\end{equation}

Suppose there is adjoint variable $\lambda(t)$ which makes the following form holds 
\begin{equation}
\int_{t_0}^{t_N}\lambda(t)\bigl(\frac{d}{dt}-\frac{\pl f}{\pl z}\bigr)z_\theta(\theta,t)  dt. 
\end{equation}

Integrate by parts, we have 
\begin{align}
\int_{t_0}^{t_N}\lambda(t)\bigl(\frac{d}{dt}-\frac{\pl f}{\pl x}\bigr)z_\theta(\theta,t)  dt & = \lambda \frac{\pl z}{\pl \theta}|_{t_0}^{t_N} + \int_{t_0}^{t_N}(-\frac{d\lambda}{dt}-\lambda\frac{\pl f}{\pl z})z_\theta dt\\
& = \int_{t_0}^{t_N}\frac{\pl z}{\pl \theta}(-\frac{d\lambda}{dt}-\lambda\frac{\pl f}{\pl z}) dt.
\end{align}
Here we have used the fact $\frac{\pl z(t_0)}{\pl \theta}=0$ in the neural network and we assumed that $\lambda(t_N)=0$.

Let 
\begin{equation}
-\frac{d\lambda}{dt}-\lambda\frac{\pl f}{\pl z} = \sum_i^N\frac{\pl L}{\pl z(t_i)}\delta(t-t_i),
\end{equation}
and we have the simplified form of the loss gradient 
\begin{equation}
\frac{\pl L}{\pl \theta} = \int_{t_0}^{t_N} \lambda \frac{\pl f}{\pl \theta} dt
\end{equation}
where the adjoint variable $\lambda(t)$ satisfies the following equation 
\begin{equation}
-\frac{d \lambda}{d t} = \lambda\frac{\pl f}{\pl z} + \sum_i^N\frac{\pl L}{\pl z(t_i)}\delta(t-t_i) 
\end{equation}
with the initial value $\lambda(t_N)=0$. 
Further the adjoint variable have the equivalent  form  
\begin{equation}
-\frac{d \lambda}{d t} = \lambda\frac{\pl f}{\pl x} 
\end{equation}
with initial value $\lambda^{initial}(t_i) = \lambda(t_i)+\frac{\pl f}{\pl x(t_i)}$ where $i = N, N-1,\cdots, 1$.  This corresponds to the $\mathcal{T}^{-1}=-\frac{d}{dt}$ when $\mathcal{T}=\frac{d}{dt}$.

\subsection{Operator adjoint methods for discrete form of neural ODE}
In the last subsection, we have given the adjoint form of the neural ODE, but they are in the continuous form. In numerical experiments of neural differential networks, the forward and backward propagations are actually done in the discrete form. So, if we compute the forward propagation by a discrete form of the neural ODE and compute the alongside back propagation with another discrete form of the adjoint equations, then the difference of discrete points and discrete schemes we choose will certainly affect the performance of the training procedure which is based on the gradient decent of the loss.  

The forward propagation is actually a discrete form of the ordinary equations, why do we not perform adjoint operator analysis on the same discrete form in the back propagation for an exact form of the adjoint getting ride of both the redundancy of the direct propagation and the inaccuracy of the adjoint forms of the discrete continuous case? In this subsection, we directly perform the analysis on the discrete form of the forward propagation and reveal the corresponding discrete adjoint forms.

\begin{remark}
 Be careful with the influence of the boundary. For implicit numerical form, it is needed to discuss detailedly.   
\end{remark}

The discrete form of the optimization problem gives
\begin{equation}
\label{eq.dl}
\begin{aligned}
\min_{\theta} \quad &L(z(t_{n_1}),\cdots, z(t_{n_i}),\cdots, z(t_{n_M}))\\
\textrm{s.t.} \quad & P_nz_n = \Delta t f(z_n,\theta,t), \\
\end{aligned}
\end{equation}
where $P_n$ is the linear combination of $k$-shift operators $T_k$, i.e., 
\begin{equation}
P_n = \sum_{k=1}^K \alpha_kT_k, \quad \text{with} \quad T_k z_n = z_{n+k}.
\end{equation}

Thus the derivatives of the loss with respect to the learning parameter reads
\begin{align}
\frac{\pl L}{\pl \theta} &=  \sum_{i=1}^M\frac{\pl L}{\pl z_{n_i}}
\frac{\pl z_{n_i}}{\pl \theta}\\
&=\sum_{i=1}^M \frac{\pl L}{\pl z_{n_i}}\sum_{n=1}^N\frac{\pl z_n}{\pl \theta}\delta_{n,n_i}\\
&=\sum_{n=1}^N\sum_{i=1}^M \frac{\pl L}{\pl z_{n_i}}\delta_{n,n_i}\frac{\pl z_n}{\pl \theta}.
\end{align}

Take the derivatives of $\theta$ on the both side of the constraint as before and we have 
\begin{equation}
(P_n-\Delta t \frac{\pl f}{\pl z_n})\frac{\pl z_n}{\pl \theta} = 
\Delta t \frac{\pl f}{\pl \theta}.
\end{equation}

Applying the operator adjoint strategy, we have
\begin{align}
\sum_{n=1}^N\sum_{i=1}^M \frac{\pl L}{\pl z_{n_i}}
\delta_{n,n_i}\frac{\pl z_n}{\pl \theta} &= \sum_{n=1}^N \lambda_n(P_n-\Delta t\frac{\pl f}{\pl z_n})\frac{\pl z_n}{\pl \theta}\\
&=\sum_{n=1}^N\lambda_n\Delta t\frac{\pl f}{\pl \theta}.
\end{align}
Thus we have 
\begin{equation}
\frac{\pl L}{\pl \theta} = \sum_{n=1}^N\lambda_n\Delta t\frac{\pl f}{\pl \theta},
\end{equation}
where the adjoint variable $\lambda_n, n=1,2,\cdots,N$ is given by 
\begin{equation}
(P_n^{-1}-\Delta t\frac{\pl f}{\pl z_n})\lambda_n = \sum_{i=1}^M\frac{\pl L}{\pl z_{n_i}}\delta_{n,n_i}.
\end{equation}
Further by the similar analysis with the continuous case, we have the adjoint equation for $\lambda_n$
\begin{equation}
P_n^{-1}\lambda_n= \Delta t\frac{\pl f}{\pl z_n},\end{equation}
with the initial value $\lambda^{initial}_{n_i}=\lambda_{n_i}+\frac{\pl L}{\pl z_{n_i}}$ and $\lambda(N)=0$ for $n = N,\cdots, n_i, \cdots, n_1$. 
Here it is easy to check that $P_n^{-1} = \sum_{k=1}^K \alpha_kT_{-k}.$

\subsection{The gaps between the original adjoints and the direct back propagation with auto differentiation.}

There are gaps of performance between the direct back propagation and the adjoint forms, see (D) and (F) in Figure \ref{fig.sd}.

The neural ODE is a continuous form in the analysis, but the forward propagation is actually done in the discrete form, i.e., when we solve a neural ode, the chosen ODE solver is applied thus the structure of the forward propagation scheme is actually integrating the discrete neural ODE, as shown in (A) and (C) of Figure \ref{fig.sd}.  On the backward propagation, the adjoint form called adjoint ODE (B) is derived by the analysis of the neural ODE in the continuous form. The adjoint ODE (B) is the adjoint form of the neural ODE (A) but they are both in the continuous form. Then traditional algorithm solves adjoint ode (B) by its discrete adjoint ode (D). There arises a question: if (D) gives the right results of the adjoint form of (C)? To answer this question, we have provided the adjoint form (E) for the discrete neural ODE (C) in the last section. From the form of (E), we can tell that (E) is the form of (D) but requiring (D) has the same discrete scheme with (C).   Meanwhile if (C) and (D) are discretized in different ways, then (D) is no longer the right adjoint form of (C) even if its continuous form (B) is the adjoint form of (A)! 

In the traditional computation, (C) is computed and (D) is used as its adjoint form, that is certainly not right, as (D) is not always the adjoint form of (C) only if they have the same discrete scheme. In fact the direct back propagation form (F) is the right derivation of the loss gradient of (C).  (D) gives the same result with (F) if and only if (D) and (C) have the same discrete scheme. Here we must emphasize that the same discrete scheme means not only the same scheme to discrete the continuous equation but also the same iteration points along the integral. The results can be easily reduced from the last section.

\begin{figure}[htp]
\includegraphics[width=\textwidth]{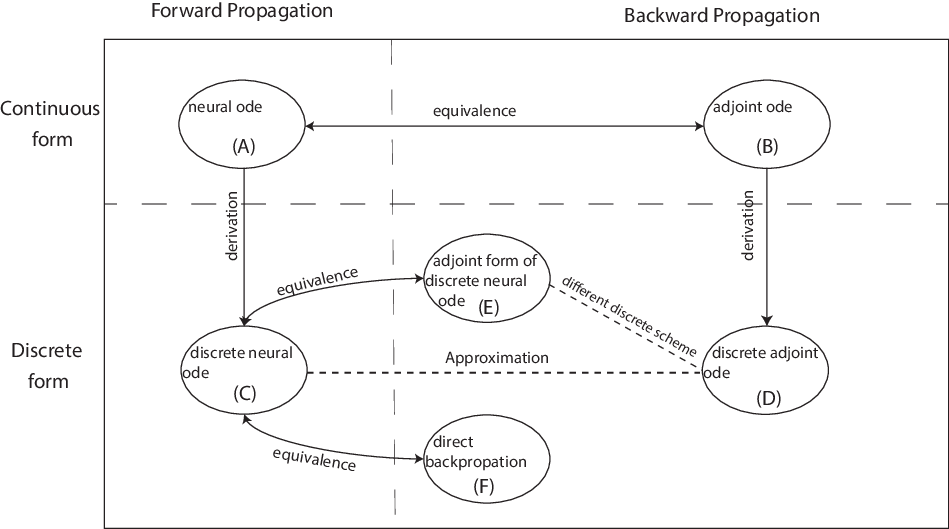}
\caption{The schemes' differences.}
\label{fig.sd}
\end{figure}


\section{Conclusion}
In this note, we mainly use the calculus of variation widely used in optimal control theory to obtain the rigorous proof of the adjoint method for the control problem of neural ode. We show that the adjoint method gives the gradients of the Lagrange cost $J(\theta)$ not the cost $L(z(t_f))$. Further, we show that the adjoint variable can be simply represented by the gradients of cost $L(z(t_f))$ to the trajectory variable $z(t)$. It is a typical property because of the control parameter is a constant with time. Besides, we proposed a simple proof of the adjoint method by a direct differential calculations. The simple but not rigorous proof shows that the derivation can be much simpler if strictness is not needed.

\section{Acknowledgement}
The author wishes to extend heartfelt gratitude to Prof. Zuoqiang Shi for the insightful discussions related to this note. Additionally, the author appreciates all readers and encourages anyone who finds errors to report them via email at pisquare@microsoft.com or hpp1681@gmail.com. Should you find this note beneficial, the author would be pleased if you could cite it in your work.

\bibliography{adjointreferences}
\bibliographystyle{unsrt}

\end{document}